\documentclass[a4paper, 11pt]{amsart}
\usepackage{amssymb,amsmath}
\usepackage[all]{xy}
\usepackage[top=30truemm,bottom=30truemm,left=25truemm,right=25truemm]{geometry}

\makeatletter
\@namedef{subjclassname@2010}{%
  \textup{2010} Mathematics Subject Classification}
\makeatother

\newcommand{\Q}{\mathbb{Q}}
\newcommand{\F}{\mathbb{F}}

\newcommand{\Spec}{\operatorname{Spec}}

\newcommand{\Der}{\operatorname{Der}}

\newcommand{\Aut}{\operatorname{Aut}}
\newcommand{\Exp}{{\rm Exp}}
\newcommand{\ri}{i}
\newcommand{\rii}{i\hspace{-.1em}i}
\newcommand{\riii}{i\hspace{-.1em}i\hspace{-.1em}i}

\newtheorem{thm}{Theorem}[section]
\newtheorem{prop}[thm]{Proposition}

\newtheorem{cor}[thm]{Corollary}
\newtheorem{defn}[thm]{Definition}
\newtheorem{example}[thm]{Example}
\newtheorem{remark}[thm]{Remark}

\begin{document}
\title[Higher derivations of Jacobian type in positive characteristic]{Higher derivations of Jacobian type in positive characteristic}
\author{Takanori Nagamine}
\address[T. Nagamine]{Graduate School of Science and Technology, Niigata University, 8050 Ikarashininocho, Nishi-ku, Niigata 950-2181, Japan}
\email{t.nagamine14@m.sc.niigata-u.ac.jp}
\date{\today}
\thanks{Research of the author was partially supported by Grant-in-Aid for JSPS Fellows  (No.\ 18J10420) from Japan Society for the Promotion of Science}

\begin{abstract}
In this paper, we study higher derivations of Jacobian type in positive characteristic. We give a necessary and sufficient condition for $(n-1)$-tuples of polynomials to be extendable in $R[x_1, \ldots, x_n]$ over an integral domain $R$ of positive characteristic. In particular, we give characterizations of variables and univariate polynomials by using the terms of higher derivations of Jacobian type in the polynomial ring in two variables over a field of positive characteristic.   
\end{abstract}

\maketitle

\setcounter{section}{-1}

\section[Introduction]{Introduction}

In this paper, we study higher derivations of Jacobian type in positive characteristic. In the case where the characteristic of the ground field is zero, derivations of Jacobian type are well known and they are one of the most important tools for understanding polynomial rings. See e.g., \cite{Mak98}, \cite{Ess00} and \cite{Fre17}. However, in the case where the characteristic of the ground field is positive, there are no concepts corresponding to derivations of Jacobian type. 

In \emph{Section 1}, we recall some kinds of higher derivations and their properties. Also we recall variables,  univariate polynomials and extendable $(n-1)$-tuples of polynomials.   

In \emph{Section 2}, we recall some properties of a smooth extension of rings. In \emph{Definition \ref{def:2.4}}, we introduce concepts for higher derivations of Jacobian type. We show that smooth ring extensions guarantee the existence of higher derivations of Jacobian type (\emph{Proposition \ref{prop:2.5}}). The main result in this paper is \emph{Theorem \ref{thm:2.8}} which gives a necessary and sufficient condition for $(n-1)$-tuples of polynomials to be extendable by the terms of higher derivations of Jacobian type. This is a generalization of \cite[Proposition 2.3]{Ess95} in positive characteristic. 
   
In \emph{Section 3}, we study higher derivations of Jacobian type on $k[x,y]$. In \emph{Theorem \ref{thm:3.1}} and \emph{Corollary \ref{cor:3.2}}, we give characterizations of variables and univariate polynomials by using the terms of higher derivations of Jacobian type. \emph{Theorem \ref{thm:3.1}} is a generalization of \cite[Theorem 3.2]{ER04} in the case where the characteristic of the ground field is positive.    

\section[Preliminaries]{Preliminaries}

Let $R$ be an integral domain of characteristic $p \geq 0$. For a positive integer $n \geq 1$, we denote $R^{[n]}$ by the polynomial ring in $n$ variables over $R$ and $Q(R)$ by the field of fractions. 

Through this section, assume that $B$ is an integral domain containing $R$. 
Let $D = \{ D_{\ell}\}_{\ell = 0}^{\infty}$ be a family of $R$-linear maps $D_{\ell} : B \to B$ for $\ell \geq 0$. We say that $D$ is a {\bf higher $R$-derivation} on $B$ if, for $f, g \in B$ and $\ell \geq 0$, 
	\begin{enumerate}
	  \item[{\rm (a)}]
	  $D_0 = {\rm id}_B$, 
	  \item[{\rm (b)}]
	  $\displaystyle D_{\ell}(fg) = \sum_{i + j = \ell}D_i(f)D_j(g)$. 
	\end{enumerate}
Note that, for a higher $R$-derivation $D = \{ D_{\ell}\}_{\ell = 0}^{\infty}$, $D_1$ is an $R$-derivation on $B$. 

For a higher $R$-derivation $D = \{ D_{\ell}\}_{\ell = 0}^{\infty}$ on $B$, we define the map $\varphi_D : B \to B[[t]]$, where $B[[t]]$ is the formal power series ring in one variable over $B$, by 
	\[  
	  \varphi_D(f) = \sum_{i = 0}^{\infty} D_i(f) t^i
	\]
for $f \in B$. The above condition (b) implies that $\varphi_D$ is a homomorphism of $R$-algebras, condition (a) implies that $\varphi_D(f) |_{t = 0} = f$. We call the mapping $\varphi_D$ the {\bf homomorphism associated to $D$}. We denote $B^D$ by the intersections of the kernel of $D_\ell$ for $\ell \geq 1$, that is, 
	\[
	  B^D = \bigcap_{\ell \geq 1} \ker D_{\ell}. 
	\]
We say that $D$ is {\bf trivial} if $B^D = B$. A higher $R$-derivation $D = \{ D_{\ell}\}_{\ell = 0}^{\infty}$ on $B$ is {\bf locally finite} if $D$ satisfies{\rm :} 
	\begin{enumerate}
	  \item[{\rm (c)}]
	  for any $f \in B$, there exists a positive integer $N_f \geq 1$ such that $D_{\ell}(f) = 0$ for any $\ell \geq N_f$, 
	\end{enumerate}
and is {\bf iterative} if $D$ satisfies{\rm :}
	\begin{enumerate}
	  \item[{\rm (d)}]
	  $\displaystyle D_i \circ D_j = \binom{i + j}{j}D_{i + j}$ for any $i, j \geq 0$. 
	\end{enumerate}
When $D = \{ D_{\ell}\}_{\ell = 0}^{\infty}$ satisfies the above conditions (a), (b), (c) and (d), we say $D$ is a {\bf locally finite iterative higher $R$-derivation}, for short an {\bf lfihd}.  

Let $D$ be an lfihd on $B$. An element $s$ of $B$ is called a {\bf local slice} of $D$ if it
satisfies the following conditions:
	\begin{enumerate}
	  \item[{\rm (a)}]
	  $s\not\in B^D$, 
	  \item[{\rm (b)}] 
	  $\deg_{t}(\varphi_D(s))= {\rm min} \{ \deg_{t}(\varphi_D(b)) \ | \ b \in B\setminus B^D \}$. 
	\end{enumerate} 
Here, we note that every nontrivial lfihd on $B$ has local slices. A local slice $s\in B$ of $D$ is called a {\bf slice} if the leading coefficient of $\varphi_D(s)$ is a unit of $B$.

\begin{prop} \label{prop:1.1} 
{\rm (cf.\ \cite[Lemma 1.4]{Miy78})}
Let $D$ be an lfihd on $B$. If $D$ has a slice $s\in B$, then $B=B^D[s]$ and $s$ is indeterminate over $B^D$. 
\end{prop}

In the rest of this section, we assume that $B= R[x_1, \ldots, x_n]\cong_R R^{[n]}$ is the polynomial ring in $n$ variables over $R$. Let $f \in B\setminus R$ be a non-constant polynomial. $f$ is called a {\bf variable} (or {\bf coordinate}) over $R$ if $R[f]^{[n-1]}=B$. Finally, $f$ is called  {\bf univariate} over $R$ if there exists a variable $g \in B$ such that $f \in R[g]$. An $(n-1)$-tuple polynomials $f_1, \ldots, f_{n-1} \in B$ is said to be {\bf extendable} if $R[f_1, \ldots, f_{n-1}]^{[1]}=B$. 

\section[Higher derivations of Jacobian type in positive characteristic]{Higher derivations of Jacobian type in positive characteristic}

First of all, we prepare some notation and results of general commutative ring theory. Let $A$ be a commutative ring and let $B$ be a commutative $A$-algebra via a homomorphism $\varphi : A \to B$. We say that $B$ is {\bf smooth} over $A$ if for any $A$-algebra $C$ with $g : A \to C$, an ideal $N \subset C$ with $N^2 = 0$ and a homomorphism of $A$-algebras $u : B \to C/N$, there exists a homomorphism of $A$-algebras $v : B \to C$ such that $v \circ \pi = u$, where $\pi : C\to C/N$ is the natural homomorphism. That is, $v$ commutes the following diagram:      
	\[
		\xymatrix
		{
		  B  \ar[r]^{\hspace{-2.8mm}u} \ar@{-->}[rd]^{\exists v} & C/N    \\
		  A \ar[u]^{\varphi} \ar[r]_{g} & C  \ar[u]_{\pi} . 
		}
	\]
For $\mathfrak{p} \in \Spec A$, we denote the residue field by $\kappa(\mathfrak{p}) = A_{\mathfrak{p}}/\mathfrak{p}A_{\mathfrak{p}}$. 

\begin{prop} \label{prop:2.1}
Let $\varphi : A \to B$ be a homomorphism of commutative rings. For $\mathfrak{p} \in \Spec A$, let $\iota_{\mathfrak{p}} : \kappa(\mathfrak{p}) \to B \otimes_A\kappa(\mathfrak{p})$ be the natural homomorphism of $A$-algebras. If $\varphi$ is smooth, then $\iota_{\mathfrak{p}}$ is also smooth for any $\mathfrak{p} \in \Spec A$. 
\end{prop}
\begin{proof} 
Omitted. 
\end{proof}

\begin{example} \label{ex:2.2}
{\rm 
Let $R$ be an integral domain containing a prime field and let $R[x, y]  \cong_R R^{[2]}$. Then the following assertions hold true. 
	\begin{enumerate}
	  \item[{\rm (a)}]
	  The natural inclusion $R \to R[x, y]$ is smooth. 
	  \item[{\rm (b)}]
	  For a variable $f \in R[x, y]$, the natural inclusion $R[f] \to R[x, y]$ is smooth. 
	  \item[{\rm (c)}] 
	  For $xy \in R[x, y]$, the natural inclusion $R[xy] \to R[x, y]$ is not smooth. 
	\end{enumerate}
}
\end{example}
\begin{proof}
(a) and (b) are obvious. We prove the assertion (c). 

Assume to the contrary that $\iota : R[xy] \to R[x, y]$ is smooth. Let $k$ be an algebraic closure of $Q(R)$. By \emph{Proposition \ref{prop:2.1}}, the following homomorphism $\iota_0$ is smooth: 
	\[
	  \iota_0 : k \cong_k \kappa(0) \to k[x, y] \otimes_{k[xy]}\kappa(0) \cong_k k[x, y]/(xy). 
	\]
Set $C = k[t]/(t^3)$ and $N = t^2C$, where $k[t]  \cong_k k^{[1]}$. Then $N^2 = 0$. We define $u : k[x, y]/(xy) \to C/N$ by $u(x) = u(y) = t$. Since $\iota_0$ is smooth, there exists a homomorphism $v : k[x, y]/(xy) \to C$ of $k$-algebras such that $\pi \circ v = u$, namely, $v$ commutes the following diagram: 
	\[
		\xymatrix
		{
		  k[x, y]/(xy)  \ar[r]^{\hspace{5mm}u} \ar@{-->}[rd]^{\exists v} & C/N    \\
		  k \ar[u]^{\iota_0} \ar[r]_{g} & C  \ar[u]_{\pi} . 
		}
	\]
Then $v(x) = t + at^2$ and $v(y) = t + bt^2$ for some $a, b \in k$. However, 
	\[
	  0 = v(xy) = v(x)v(y) = t^2, 
	\]
which is a contradiction. 
\end{proof}

From now on, let $R$ be an integral domain containing a prime field of characteristic $p\geq 0$ and let $B=R[x_1, \ldots, x_n]\cong_RR^{[n]}$ be the polynomial ring in $n$ variables over $R$. We denote $\partial_{x_i}$ by the partial derivative with respect to $x_i$. 

\begin{defn} \label{def:2.3}
{\rm 
Let $d$ be an $R$-linear map on $B$. For $b\in B$, we denote $[d](b)$ by the result of calculation of $d(b)$ as if the characteristic of $R$ is zero, that is, we consider $p\neq 0$ in $B$. For $\ell \geq 1$, we define, for $b\in B$, 
	\[
	[d]^{\ell}(b)=[d]\left([d]^{\ell-1}(b)\right). 
	\]
}
\end{defn}

For $f_1, \ldots, f_{n-1} \in B$, let $F=(f_1, \ldots, f_{n-1})$ and $R[F]=R[f_1, \ldots, f_{n-1}]$. Then $F$ defines the $R$-derivation $\Delta_F$ and $\widetilde{\Delta}_F$ on $B$ by, for $g\in B$, 
	\[
	  \Delta_F(g) = {\rm det}\:\left({\partial x_j}(f_i)\right)_{1\leq i,j\leq n}, \hspace{5mm} 
	  \widetilde{\Delta}_F(g)=[{\rm det}]\:\left(\left[{\partial x_j}\right](f_i)\right)_{1\leq i,j\leq n}, 
	\]
where for a matrix $A$, $[{\rm det}]\:(A)$ means the result of calculation of ${\rm det}\:(A)$ as if the characteristic of $R$ is zero. $\widetilde{\Delta}_F$ is called the {\bf Jacobian derivation} determined by $F$. 

\begin{defn} \label{def:2.4}
{\rm
A higher $R$-derivation $D = \{ D_{\ell} \}_{\ell = 0}^{\infty}$ on $B$ is {\bf of Jacobian type} if there exists $F=(f_1, \ldots, f_{n-1}) \in B^{n-1}$ such that 
	\begin{enumerate}
	  \item[{\rm (a)}]
	  $R[F]\cong_RR^{[n-1]}$, 
	  \item[{\rm (b)}]
	  $f_1, \ldots, f_{n-1} \in B^D$, 
	  \item[{\rm (c)}]
	  $D_1=\widetilde{\Delta}_F$, 
	  \item[{\rm (d)}]
	  $\displaystyle D_{\ell} = \frac{1}{\ell !}[D_1]^{\ell}$ for\: 
	  	$
	  		\begin{cases}
			{\rm any}\: \: \ell \geq 0 & {\rm if}\: \: p=0,\\  
			0\leq \ell\leq p-1 & {\rm if}\: \: p>0. 
			\end{cases}
		$
	\end{enumerate}
}
\end{defn} 
\noindent 
Note that the above condition (b) is equivalent to $\varphi_D(f_i) = f_i$ for $1\leq i\leq n-1$, that is, $\varphi_D$ is a homomorphism over $R[F]$. 

\begin{prop} \label{prop:2.5} 
For $f_1, \ldots, f_{n-1} \in B \setminus R[{x_1}^p, \ldots, {x_n}^p]$, let $F=(f_1, \ldots, f_{n-1})$. Suppose that $R[F]\cong_RR^{[n-1]}$. If the natural inclusion $R[F] \to B$ is smooth, then there exists a higher $R$-derivation $D$ on $B$ of Jacobian type determined by $F$. 
\end{prop}
\begin{proof}
Define $D_0 = {\rm id}_B$ and $D_1 = \widetilde{\Delta}_F \:(\neq 0)$. Let $B[t] \cong_B B^{[1]}$. Here, we define a map $\varphi_{\ell} : B \to B[t]/(t^{\ell + 1})$ by, for $g \in B$ and $1 \leq \ell \leq p - 1$,  
	\[
	  \varphi_{\ell}(g) = \sum_{i = 0}^{\ell}\frac{1}{i !}{[\widetilde{\Delta}_F}]^i(g)t^i. 
	\]
Then $\varphi_{\ell}$ is a homomorphism of $R[F]$-algebras such that $\varphi_{\ell}(g) |_{t = 0} = g$ for any $g \in B$. 

For $r \geq 0$, let $C_r = B[t]/(t^{p + r})$ and $N_r = t^{p + r - 1}C_r$. Then 
	\[
	N_r^2 = 0, \: C_r/N_r \cong_B B[t]/(t^{p + r  - 1}). 
	\]
Since $R[F] \to B$ is smooth, there exists a homomorphism $\varphi_{p + r} : B \to C_r$ of $R[F]$-algebras such that $\pi_r \circ \varphi_{p + r} = \varphi_{p + r - 1}$, that is, we have the following diagram: 
	\[
		\xymatrix
		{
		  B  \ar[r]^{\hspace{-10mm} \varphi_{p + r - 1}} \ar@{-->}[rd]^{\hspace{1mm} \exists \varphi_{p + r}} & 
		  B[t]/(t^{p + r  - 1})  \ar[r]^{\hspace{5mm} \cong_B}&
		   C_r/N_r  \\
		  R[F] \ar[u]^{} \ar[r]_{} & 
		  B[t]/(t^{p + r}) \ar[r]^{\hspace{5mm} =} &
		  C_r \ar[u]_{\pi_r} . 
		}
	\]
Moreover $\varphi_{p + r}(g) |_{t = 0} = g$ for any $g \in B$. For $0 \leq i \leq r - 1$, by using $\varphi_{p + r}$, we define a homomorphism of $R$-modules $D_{p + i} : B \to B$ by the following formula:  
	\[ 
	  \varphi_{p + r}(g) = \sum_{\ell = 0}^{p - 1}\frac{1}{\ell !}{[\widetilde{\Delta}_F]}^{\ell}(g)t^{\ell} + \sum_{i = 0}^{r - 1}D_{p + i}(g)t^{p + i}
	\]
for $g \in B$. By constructing such homomorphisms inductively, we have a homomorphism of $R[F]$-algebras $\varphi = \varphi_{\infty} : B \to B[[t]] \cong_BB^{[[1]]}$ such that, for $g \in B$, $\varphi(g) |_{t = 0} = g$ and   
	\[
	  \varphi(g) = \sum_{\ell = 0}^{p - 1}\frac{1}{\ell !}{[\widetilde{\Delta}_F]}^{\ell}(g)t^{\ell} + \sum_{i = 0}^{\infty}D_{p + i}(g)t^{p + i}. 
	\]
Set $D_{\ell} = {\ell !}^{-1}{[\widetilde{\Delta}_F]}^{\ell}$ for $0 \leq \ell \leq p - 1$ and $D = \{ D_{\ell} \}_{\ell = 0}^{\infty}$. By the construction of each $D_{\ell}$, we see that $D$ is a higher $R$-derivation on $B$ of Jacobian type determined by $F$. 
\end{proof} 

We note that smoothness is not necessarily for the existence of higher $R$-derivations of Jacobian type. We give an example below. 

\begin{example} \label{ex:2.6}
{\rm
Let $f=xy\in \F_2[x,y]\cong_{\F_2}\F_2^{[2]}$. We define $\varphi : \F_2[x,y]\to\F_2[x,y][[t]]$ by 
	\[
	  \varphi(x)=\sum_{\ell=0}^{\infty}xt^{\ell}, \: \varphi(y)=y+yt. 
	\]
Then there exists a higher $R$-derivation $D=\{ D_{\ell}\}_{\ell=0}^{\infty}$ on $\F_2[x,y]$ such that $\varphi_D=\varphi$ and $D_1=\widetilde{\Delta}_f$. Moreover, $xy\in\F_2[x,y]^D$. Therefore $D$ is a higher $R$-derivation of Jacobian type determined by $f$, but the natural inclusion $\F_2[xy]\to\F_2[x,y]$ is not smooth. 
}
\end{example}

In order to explain the statement of the main theorem (\emph{Theorem \ref{thm:2.8}}), we introduce some definitions as below. For a positive integer $\ell \geq 1$, we write $\ell ! = p^{e(\ell)}m_{\ell}$, where $p$ does not divide $m_{\ell}$. Let $\ell \geq 1$ and $g \in B$. For a non-zero $R$-derivation $d \in \Der_RB$, we say that ${\ell !}^{- 1}[d]^{\ell}$ is {\bf defined} at $g$ if, for any $1\leq i \leq \ell$, there exists $g_{i} \in B$ such that 
	\[
	  [d]^{i}(g)=p^{e(i)}g_{i}. 
	\]
We define its value by $({\ell !}^{- 1}[d]^{\ell})(g) = m_{\ell}^{-1}g_{\ell}$. When ${\ell !}^{- 1}[d]^{\ell}$ is defined at any $g\in B$ and $\ell \geq 1$, we consider the map $\Exp(td) : B \to B[[t]] \cong_BB^{[[1]]}$ defined by  
	\[
	  \Exp (td) (g) = \sum_{\ell = 0}^{\infty} \frac{[d]^{\ell}}{\ell !}(g)t^{\ell} = \sum_{\ell = 0}^{\infty} \frac{1}{m_{\ell}}g_{\ell}t^{\ell}.  
	\]
By the definition of $\Exp (td)$, we see that it is a homomorphism of $R$-algebras and satisfies that $\Exp (td)(g) |_{t = 0} = g$ for $g \in B$. Therefore, there exists a higher $R$-derivation $D$ on $B$ such that $\Exp (td)=\varphi_D$. 

In order to check whether the map $\Exp(td)$ is defined or not, it is enough to show that ${\ell !}^{- 1}[d]^{p^{\ell}}$ is defined at $x_1, \ldots, x_n$ for any $\ell \geq 1$.  

\begin{example} \label{ex:2.7}
{\rm 
Let $B = \F_3[x, y] \cong_{\F_3} \F_3^{[2]}$. Set $d_1 = \widetilde{\Delta}_{x-y^3}= 3y^2\partial_x + \partial_y$ and $d_2 = \widetilde{\Delta}_{xy}=-y\partial_x + x\partial_y$. Then $\Exp(td_1)$ is defined, but $\Exp(td_2)$ is not defined. Indeed, for $d_1$, it is clear that ${\ell !}^{- 1}{[d_1]}^{\ell}$ is defined at $y$ for $\ell \geq 1$. Also, it is defined at $x$ for $\ell \geq 1$ as the following table. 
\begin{table}[htb]
	\begin{tabular}{c||c|c|c}
		$\ell$& $\ell ! = p^{e(\ell)}m_{\ell}$ & ${[d_1]}^{\ell}(x)$ &$({\ell !}^{- 1}{[d_1]}^{\ell})(x)$ \\ \hline \hline
		$1$ & $3^0\cdot 1$& $3y^2$ & $0$ \\ \hline 
		$2$ & $3^0\cdot 2$ & $3\cdot 2y$ & $0$ \\ \hline 
		$3$ & $3^1\cdot 2!$ & $3!$ & $1$ \\ \hline 
		$\ell \geq 4$ & $\ell !$ & $0$ & $0$ 
	\end{tabular}
\end{table}

\noindent 
Therefore $\Exp(td_1)$ can be defined and $\Exp(td_1)(x)=x+t^3$, $\Exp(td_1)(y)=y+t$. Moreover, it is easy to show that $\Exp(td_1)(x-y^3)=x-y^3$. 

On the other hand, for $d_2$, 
	\[
	  {[d_2]}^{\ell}(x) = 
	  	\begin{cases}
		(-1)^{\frac{\ell+1}{2}}y & {\rm (\ell \: \: is\: odd)},\\
		(-1)^{\frac{\ell}{2}}x & {\rm (\ell \: \: is\: even)}, 
		\end{cases}
	\]
hence ${\ell !}^{- 1}{[d_2]}^{\ell}$ is not defined at $x$ when $\ell \geq 2$. 
}
\end{example}

Let $d \in \Der_RB$ be a non-zero $R$-derivation. We consider $\Aut_RB$ as a subgroup of $\Aut_RB[[t]]$ by $\sigma(t) = t$ for $\sigma \in \Aut_RB$. If $\Exp(td)$ can be defined, then $\Exp (t \cdot {}^{\sigma}\hspace{-.2em}d)$ can be defined for any $\sigma \in \Aut_RB$, where ${}^{\sigma}\hspace{-.2em}d := \sigma^{-1} \circ d \circ \sigma$. In particular, the following holds: 
	\[
	  \Exp (t \cdot {}^{\sigma}\hspace{-.2em}d) = \sigma^{-1} \circ  \Exp (td) \circ \sigma.
	\] 

The following is the main result in this paper which is a generalization of \cite[Proposition 2.3]{Ess95} in positive characteristic.  

\begin{thm} \label{thm:2.8} 
Let $R$ be an integral domain of characteristic $p > 0$ and let $B=R[x_1, \ldots, x_n]\cong_RR^{[n]}$ be the polynomial ring in $n$ variables over $R$. For $f_1, \ldots, f_{n-1} \in B$, let $F=(f_1, \ldots, f_{n-1})$. Then the following conditions are equivalent{\rm :}
	\begin{enumerate}
	  \item[{\rm (\ri)}]
	  $F$ is extendable. 
	  \item[{\rm (\rii)}]
	  $D = \{ {\ell !}^{- 1}{[\widetilde{\Delta}_F]}^{\ell} \}_{\ell = 0}^{\infty}$ can be defined, is an lfihd  on $B$ of Jacobian type determined by $F$ such that $B^D = R[F]$ and has a slice. 
	\end{enumerate}
\end{thm}
\begin{proof}
{\bf (\ri) $\Longrightarrow$ (\rii)}\: 
Since $F$ is extendable, there exists $s \in B$ such that $R[f_1, \ldots, f_{n-1}, s] = B$. Define the $R$-automorphism $\sigma : B \to B$ by $\sigma(x_i) = f_i$ for $1\leq i \leq n-1$ and $\sigma(x_n) = s$. We may assume that $\widetilde{\Delta}_F(s) = 1$ and ${}^{\sigma}\hspace{-.2em}\widetilde{\Delta}_F = \partial_{x_n}$. It is clear that $\Exp(t\partial_{x_n})$ can be defined, hence $\Exp (t \cdot {}^{\sigma}\hspace{-.2em}\widetilde{\Delta}_F) = \sigma^{-1} \circ  \Exp (t\partial_{x_n}) \circ \sigma$. This implies that ${\ell !}^{- 1}{[\widetilde{\Delta}_F]}^{\ell}$ is defined at any $g \in B$ and $\ell \geq 1$. Set $D = \{ {\ell !}^{- 1}{[\widetilde{\Delta}_F]}^{\ell} \}_{\ell = 0}^{\infty}$. Then $D$ is an lfihd on Jacobian type determined by $F$. It is clear that $B^D = R[F]$ and $s$ is a slice of $D$. 

{\bf (\rii) $\Longrightarrow$ (\ri)}\: 
Let $s\in B$ be a slice of $D$. By \emph{Proposition \ref{prop:1.1}}, $B=R[F][s]$, which implies that $F$ is extendable. 
\end{proof}

\section{Higher derivations of Jacobian type on $k[x,y]$}

Let $k$ be a field of characteristic $p > 0$. Through this section, we suppose that $k[x,y]\cong_kk^{[2]}$ is the polynomial ring in two variables over $k$. 

By using \emph{Theorem \ref{thm:2.8}}, we have the following result. This is a generalization of \cite[Theorem 3.2]{ER04} in the case where the characteristic of the ground field is positive.  

\begin{thm} \label{thm:3.1} 
Let $f \in k[x,y]$. Then the following conditions are equivalent{\rm :}
	\begin{enumerate}
	  \item[{\rm (\ri)}]
	  $f$ is a variable. 
	  \item[{\rm (\rii)}]
	  $D = \{ {\ell !}^{- 1}{[\widetilde{\Delta}_f]}^{\ell} \}_{\ell = 0}^{\infty}$ can be defined, is an lfihd  on $k[x,y]$ of Jacobian type determined by $f$ such that $B^D = k[f]$ and has a slice. 
	  \item[{\rm (\riii)}]
	  $D = \{ {\ell !}^{- 1}{[\widetilde{\Delta}_f]}^{\ell} \}_{\ell = 0}^{\infty}$ can be defined and is an lfihd  on $k[x,y]$ of Jacobian type determined by $f$ such that $B^D = k[f]$. 
	\end{enumerate}
\end{thm}
\begin{proof}
{\bf (\ri) $\Longrightarrow$ (\rii)}\: 
This implication follows from \emph{Theorem \ref{thm:2.8}}. 

{\bf (\rii) $\Longrightarrow$ (\riii)}\: 
Obvious. 

{\bf (\riii) $\Longrightarrow$ (\ri)}\: 
Since the ring $k[f]$ is the kernel of the lfihd $D$, it follows from \cite[Theorem 1]{Koj14} that $f$ is a variable. 
\end{proof}

By using \emph{Theorem \ref{thm:3.1}}, we have the following. This is a generalization of \cite[Corollary 4.6]{Fre17} in the case where the characteristic of the ground field is positive.  

\begin{cor} \label{cor:3.2}
For $f \in k[x,y] \setminus k[x^p, y^p]$, the following two conditions are equivalent{\rm :}
	\begin{enumerate}
	  \item[(\ri)]
	  $f$ is univariate. 
	  \item[(\rii)]
	  $D = \{ {\ell !}^{- 1}{[\widetilde{\Delta}_f]}^{\ell} \}_{\ell = 0}^{\infty}$ can be defined and is an lfihd  on $k[x,y]$ of Jacobian type determined by $f$. 
	\end{enumerate} 
\end{cor}
\begin{proof}
{\bf (\ri) $\Longrightarrow$ (\rii)}\: 
Since $f$ is univariate, there exists a variable $g \in k[x,y]$ such that $f \in k[g]$. Then $f=u(g)$ for some $u(t)\in k[t]\cong_kk^{[1]}$. Hence $\widetilde{\Delta}_f=u'(g)\widetilde{\Delta}_g$, where $u'(t)$ is the derivative of $u(t)$ with respect to $t$. 
By \emph{Theorem \ref{thm:3.1}}, $\delta=\{\ell^{-1}[\widetilde{\Delta}_g]^{\ell}\}_{\ell=0}^{\infty}$ is an lfihd on $k[x,y]$ of Jacobian type determined by $g$ with $k[x,y]^{\delta} = k[g]$. Since $u'(g) \in k[x,y]^{\widetilde{\Delta}_g}$, we have $[\widetilde{\Delta}_f]^{\ell}=u'(t)^{\ell}[\widetilde{\Delta}_g]^{\ell}$ for any $\ell \geq 1$. Therefore, ${\ell !}^{-1}[\widetilde{\Delta}_f]=u'(t)^{\ell}{\ell !}^{-1}[\widetilde{\Delta}_g]^{\ell}$ is defined. Here, $D=\{{\ell !}^{-1}[\widetilde{\Delta}_f]^{\ell}\}$ is a well-defined iterative higher $R$-derivation of Jacobian type determined by $f$. Furthermore, since $\delta$ is locally finite, so is $D$. 

{\bf (\rii) $\Longrightarrow$ (\ri)}\: 
Since $D$ is locally finite and iterative, it follows from \cite[Theorem 1]{Koj14} that $k[x,y]^D=k[g]$ for some variable $g\in k[x,y]$. Since $D$ is of Jacobian type, we have $f\in k[x,y]^D=k[g]$. This implies that $f$ is univariate. 
\end{proof}

\begin{remark}
{\rm
In Theorems \ref{thm:2.8} and \ref{thm:3.1}, we need to consider $\widetilde{\Delta}_F$, not ${\Delta}_F$. 

For example, we consider $f=x-y^3\in\F_3[x,y]\cong_{\F_3}\F_3^{[2]}$ (see also Example \ref{ex:2.7}). Then $\widetilde{\Delta}_f=3y^2\partial_x+\partial_y$ and $\Delta_f=\partial_y$. Clearly, both of $D_1:=\{ {\ell !}^{- 1}{[\widetilde{\Delta}_f]}^{\ell} \}_{\ell = 0}^{\infty}$ and $D_2:=\{ {\ell !}^{- 1}{[{\Delta}_f]}^{\ell} \}_{\ell = 0}^{\infty}$ are well-defined, locally finite and iterative. However, since $\varphi_{D_2}$ is defined by 
	\[
	\varphi_{D_2}(x)=x, \:\: \varphi_{D_2}(y)=y+t, 
	\]
we have $\varphi_{D_2}(x-y^3)=x-y^3-t^3$. Therefore $f\not\in\F_3[x,y]^{D_2}$, which implied that $D_2$ is not of Jacobian type determined by $f$. 
}
\end{remark}


\end{document}